\documentclass[11pt, a4paper,leqno]{amsart}
\usepackage{amsmath,amsthm,amscd,amssymb,amsfonts, amsbsy}
\usepackage[english]{babel}
\usepackage{latexsym}
\usepackage{txfonts}
\usepackage{exscale}
\usepackage{enumitem}
\usepackage{bbm}
\usepackage{enumitem}
\usepackage{mathabx}
\usepackage{cite}

\usepackage[colorlinks,citecolor=red,hypertexnames=false]{hyperref}
\usepackage{color}

\calclayout
\allowdisplaybreaks
\theoremstyle{plain}
\newtheorem{theorem}{Theorem}
\newtheorem{lemma}[theorem]{Lemma}

\newtheorem{proposition}[theorem]{Proposition}

\theoremstyle{definition}

\theoremstyle{remark}

\newcommand{\eps}{\epsilon}

\newcommand{\re}{\mathbb{R}}
\newcommand{\rd}{{\mathbb{R}^d}}

\newcommand{\rdd}{\mathbb{R}^{d+1}}

\newcommand{\cF}{\mathcal{F}}
\newcommand{\cS}{\mathcal{S}}
\renewcommand{\i}{\mathrm{i}}
\newcommand{\e}{\mathrm{e}}

\def\div{\mathop{\operatorname{div}}\nolimits}

\date{\today}

\author{Simon Bortz}
\author{Moritz Egert}
\author{Olli Saari}

\address{Simon Bortz, Department of Mathematics, University of Alabama, Tuscaloosa, AL, 35487, USA}
\email{sbortz@ua.edu}

\address{Moritz Egert, Universit\'{e} Paris-Saclay, CNRS, Laboratoire de math\'{e}matiques d'Orsay, 91405, Orsay, France}
\email{moritz.egert@universite-paris-saclay.fr}

\address{Olli Saari, Mathematical Institute, University of Bonn, Endenicher Allee 60, 53115 Bonn, Germany}
\email{saari@math.uni-bonn.de} 

\begin{document}
\allowdisplaybreaks

\title[Time regularity for parabolic systems of $p$-Laplace type]{Note on time-regularity for weak solutions to parabolic systems of $p$-Laplace type}

\begin{abstract}
We show that local weak solutions to parabolic systems of $p$-Laplace type are H\"older continuous in time with values in a spatial Lebesgue space and H\"older continuous on almost every time line. We provide an elementary and self-contained proof building on the local higher integrability result of Kinnunen and Lewis. 
\end{abstract}
\maketitle

\section{Introduction}
\noindent Let $d \ge 2$ and $2d/(d+2)<p<\infty$ and $N \ge 1$. Consider the following parabolic system of $p$-Laplace type:
\begin{equation}
\label{pparatype.eq}
\frac{\partial u_i}{\partial t} = \div A_i(t,x, \nabla u) + B_i(t,x,\nabla u), \quad i = 1, \dots, N, \quad \text{ in $I \times Q$},
\end{equation}
where $I \subset \mathbb{R}$ is an interval, $Q \subset \mathbb{R}^d$ a cube, and $A_i$ and $B_i$ satisfy certain structural conditions. These are the same as in \cite{KL} and do not require any smoothness of $A_i$ and $B_i$, see Section~\ref{subsec.structure}. In a celebrated paper, Kinnunen and Lewis have obtained the higher integrability of the gradient of weak solutions. 

\begin{theorem}[Theorem~2.8 in {\cite{KL}}]
\label{T2.8KL} 
There exists $\delta >0$ depending on $p$, $d$ and the structural constants $c_1$, $c_2$ and $c_3$ such that if $u \in L^2(I \times Q) \cap L^p(I;W^{1,p}(Q))$ is a weak solution to \eqref{pparatype.eq}, $I' \Subset I$ and $Q' \Subset Q$, then 
\[
u \in L^{p + \delta}(I';W^{1,p + \delta}(Q')).
\]
The norm of $u$ in $L^{p + \delta}(I';W^{1,p + \delta}(Q'))$ depends on the same constants, on $N$, $I$, $I'$, $Q$, $Q'$, the structural constant $c_4$ and the norms $\|u\|_{L^2(I \times Q)}$ and $\|u\|_{L^p(I; W^{1,p}(Q))}$.
\end{theorem}

The case $p=2$ is due to earlier work of Giaquinta and Struwe \cite{GS1982}. The significance of these results is highlighted by the otherwise lacking regularity for solutions to parabolic systems, which can be essentially discontinuous.

In this short note we prove the following in-time H\"older continuity as an addendum to the Kinnunen--Lewis result.

\begin{theorem}
\label{mainthrm.thrm} 
Let $\alpha := \frac{1}{2} (\tfrac{1}{p} - \tfrac{1}{p+\delta})$ and $q := \frac{2}{1-2\alpha}$, where $\delta>0$ is from Theorem~\ref{T2.8KL}. If $u \in L^2(I \times Q) \cap L^p(I;W^{1,p}(Q))$ is a weak solution to \eqref{pparatype.eq}, $I' \Subset I$ and $Q' \Subset Q$, then there is a representative
\[
u \in C^\alpha(I'; L^q(Q')).
\]
The norm of $u$ in $C^\alpha(I'; L^q(Q'))$ has the same dependencies as in Theorem \ref{T2.8KL}. Moreover, for a.e.\ $x \in Q'$ the restriction $u|_{I' \times \{x\}}$ is $\alpha$-H\"older continuous.
\end{theorem}

Experts in interpolation theory of vector-valued Triebel--Lizorkin type spaces (see \cite{Denk-Kaip}) will realize that Theorem~\ref{mainthrm.thrm} can be obtained from midway complex interpolation of the smoothness properties
\begin{align*}
u \in L^{p + \delta}(I';W^{1,p + \delta}(Q')) \quad \& \quad 
u \in W^{1,p'}(I';W^{-1,p'}(Q')),
\end{align*}
where the second one follows from the equation \eqref{pparatype.eq}. Since $2/q = 1/(p+\delta) + 1/p'$, we find $u \in H^{1/2,q}(I'; L^q(Q'))$. As time is a one-dimensional variable, this breaks the threshold $1-q/2 < 0$ in embeddings of such vector-valued Bessel potential spaces and leads to H\"older continuity. Still, we believe that an elementary and self-contained proof to deduce Theorem~\ref{mainthrm.thrm} from Theorem~\ref{T2.8KL} will be of interest for a broader audience and the purpose of our note is to provide such an argument. 

The abstract strategy sketched above is our guide in doing so. First, we smooth and localize the weak solution $u$ and use the equation to write the $t$-derivative of the approximant as a global negative order Bessel potential (Lemma \ref{tdervbound.lem}). Second, we use the scalar-valued Mihlin Fourier multiplier theorem and Hadamard's three lines theorem from complex analysis to bound a fractional order potential (Proposition~\ref{besselinterp.prop}). Third, a Fourier analytic characterization of H\"older continuity can be used to obtain the desired regularity of the approximant and further that of the local solution itself (Section~\ref{proof.sec}).

We close this introduction with a brief comparison to our previous work with P.~Auscher in \cite{ABES1}, where we obtained regularity as in Theorem~\ref{mainthrm.thrm} for linear operators and $p=2$ by a more involved approach. See also \cite{zaton} for a generalization to higher order systems.
The flexibility in the definition of the structure functions $A$ and $B$, allows us to use Theorem~\ref{mainthrm.thrm} for inhomogenous linear systems of the form
\[\frac{\partial u_i}{\partial t} - \div A_i(t,x, \nabla u) - B_i(t,x,\nabla u) = \div F_i + f_i, \quad i = 1, \dots, N, \quad \text{ in $I \times Q$},\]
where $F, f \in L^{2 + \eta}$ for some $\eta > 0$. The condition on $f$ is more restrictive than in \cite{ABES1}. This is needed here -- as in the classical Lions theory~\cite{Li57} -- because we use $u \in W^{1,p'}(I'; W^{-1,p'}(Q'))$ as \emph{a priori} information.
\bigskip

\noindent 
\textbf{Acknowledgement.} This research was supported by the CNRS through a PEPS
JCJC project and by DFG through DFG SFB 1060 and DFG EXC 2047.

%%%%%%
\section{Preliminaries}
\label{sec:prelim}

\subsection{Structural assumptions}
\label{subsec.structure}
We summarize the assumptions of \cite{KL}. The matrix-valued function $A: I \times Q \times (\rd)^N  \to \re^{d \times N}$ has columns given by
	\[A_i = A_i(t,x,V), \quad i = 1, \ldots, N \]
and the vector-valued function $B : I \times Q \times (\rd)^N   \to \re^N $ has scalar entries
	\[B_i = B_i(t,x, V), \quad i = 1, \ldots, N.\]
Both are (Lebesgue) $(d+1)$-measurable functions on $I \times Q$, whenever $V = V(t,x)$ is $(d+1)$-measurable on $I \times Q$. For example, $A$ and $B$ could be of Carath\'{e}odory type. We also assume there are positive constants $c_j$, $j=1,2,3$, such that for almost every $(t,x) \in I \times Q$ and every $V \in (\rd)^N $,
\begin{align*}
	|A_i(t,x, V)| &\le c_1|V|^{p-1} + h_1(t,x),\\
	|B_i(t,x, V)| &\le c_2|V|^{p-1} + h_2(t,x),
\intertext{for $i = 1,\dots N$ and}
	\sum_{i = 1}^N\langle A_i(t,x, V), V_i &\rangle  \ge c_3 |V|^p - h_3(t,x).
\end{align*}
Here $\langle \cdot, \cdot \rangle$ is the standard inner product on $\rd$ and $h_j$, $j = 1,2,3,$ are measurable functions on $I \times Q$ satisfying
\[\|(|h_1| + |h_2|)^{p/(p-1)} + |h_3|\|_{L^{\hat{q}}(I \times Q)} = c_4 < \infty,\]
for some $\hat{q} > 1$.

\subsection{Weak solutions}
\label{subsec.weak}
The space $L^{p}(I; W^{1,p}(Q))$ consists of all functions $f \in L^{p}(I \times Q)$ so that for almost every $t \in I$ the function $f(t, \cdot)$ is in the usual Sobolev space $W^{1,p}(Q)$ and 
\[\|f\|_{L^{p}(I; W^{1,p}(Q))} := \|f\|_{L^{p}(I \times Q)} + \||\nabla f|\|_{L^{p}(I \times Q)} < \infty. \]
We use the same notation for $\mathbb{R}^{N}$ valued functions with the obvious interpretation. We then say that $u$ is a weak solution to \eqref{pparatype.eq} if $u \in L^2(I \times Q) \cap L^p(I;W^{1,p}(Q))$ and if 
\[\int_{I} \int_{Q} \sum_{i=1}^N \left(-u_i\frac{\partial \phi_i}{\partial t}  +\langle A_i(t,x,\nabla u), \nabla \phi_i \rangle - B_i(t,x,\nabla u)\phi_i \right) \; dx \, dt = 0\]
holds for all $\phi = (\phi_1, \dots, \phi_N) \in C^\infty_c(I \times Q)$.

\subsection{Potential spaces}
\label{subsec.potential}
We define the Fourier transform on the Schwartz space $\cS(\rdd; \mathbb{C})$ as usual by
\begin{align*}
\cF f(\tau, \xi) = \iint e^{-\i \tau t - \i\langle \xi, x \rangle} f(t,x) \; d x \, d t
\end{align*}
and extend it to the tempered distributions by duality. The partial Fourier transforms with respect to only space or time variables are denoted by the subscripts $x$ and $t$. We define the Bessel potentials of order $s \in \mathbb{C}$ through
\begin{align*}
J^{s} f & = \cF^{-1} ( (1+ |\cdot|^{2} )^{-s/2} \cF f).
\end{align*}
Again, a subscript $x$ or $t$ tells with respect to which variable the potential is taken. If $s > 0$ and $f \in L^p(\rd)$ for some $p \in (1,\infty)$, then $J_x^s f$ is given as a convolution with an integrable function
\begin{align}
\label{potential-formula.eq}
J_x^s f = G_x^s \ast_x f, \qquad G_x^s (x) = \frac{1}{(4 \pi)^{\frac{s}{2}} \Gamma(\frac{s}{2})} \int_0^{\infty} \delta^{\frac{s-d}{2}} \e^{\frac{-\pi |x|^2}{\delta} - \frac{\delta}{4\pi}} \; \frac{d \delta}{\delta}.
\end{align}
See Section V.3.1 in \cite{Stein} for this classical formula. The Bessel potential space $H^{1,p}(\rd) = \{J_x^1 f: f \in L^{p}(\rd)\}$ with norm $g \mapsto \| J_x^{-1}g \|_{L^{p}(\rd)}$ coincides with $W^{1,p}(\rd)$ up to equivalence of norms. See Section V.3.3 in \cite{Stein}.

\subsection{Mollification}
\label{subsec.molli}
The definition of weak solutions does not imply any \emph{a priori} regularity in time direction. This causes technical problems, which in the context of this paper can be overcome through a mollification argument. Let $\varphi \in C_c^\infty(\rdd)$ be an even function with integral one that we fix from this point on. 
For $ \epsilon >0$ we denote the mollification of a function $g$ with $\varphi$ by
\[g_\epsilon(t,x) := \iint \frac{1}{\epsilon^{n+1}} \varphi \left( \frac{t - s}{\epsilon}, \frac{x-y}{\epsilon} \right) g(s,y) \; dy \, ds .  \]

%%%%%%
\section{A priori potential estimate}
\label{tdersec.sec}

\noindent We rephrase integrability and differentiability of localized weak solutions to \eqref{pparatype.eq} using Bessel potentials. The first inequality below is a reference to Theorem \ref{T2.8KL} whereas the second one expresses the regularity of the time derivative of the localized solution that follows from the equation. We call a constant \emph{admissible} if it depends on $p$, $d$, $N$, $I$, $Q$, $c_1,\ldots, c_4$, $\|u\|_{L^2(I \times Q)}$, $\|u\|_{L^p(I; W^{1,p}(Q))}$ and the chosen cut-off function $\chi \in C_c^{\infty}( I\times Q; \re)$. 

\begin{lemma}
\label{tdervbound.lem} 
Let $\chi \in  C_c^{\infty}( I\times Q )$. Let $u$ be a weak solution to \eqref{pparatype.eq} in $I \times Q$ and define $v := \chi u$. Then there is an admissible constant $C$ such that for any $\epsilon > 0$,
\begin{align*}
\| J_{x}^{-1} (v_\epsilon) \|_{L^{p+\delta}(\rdd)} &\leq C, \\
\| J_{t}^{-1} J_{x}^{1} (v_\epsilon) \|_{L^{p'}(\rdd)} &\leq C.
\end{align*}
\end{lemma}

\begin{proof}
We use the symbol $\lesssim$ for inequalities that hold up to a multiplicative admissible constant. We obtain from Young's inequality, the choice of $\chi$ and Theorem~\ref{T2.8KL} that
\begin{align*}
\|v_\epsilon\|_{L^{p+\delta}(\re; W^{1,p+\delta}(\rd))} 
\lesssim \|u\|_{L^{p+\delta}(I; W^{1,p+\delta}(Q))} 
\leq C.
\end{align*}
By coincidence of Sobolev and potential spaces, the left-hand side is comparable to $\| J_{x}^{-1} (v_\epsilon) \|_{L^{p+\delta}(\rdd)}$. Hence, we have the first estimate.

To prepare the second estimate, we fix $\phi \in \cS(\rdd; \mathbb{R}^{N})$ normalized such that $\|\phi\|_{L^p(\re; W^{1,p}(\rd))} = 1$. By H\"older's inequality we have that
\begin{align*}
\bigg|\iint \langle v_\epsilon, \phi \rangle  \; dx \, dt \bigg|
= \bigg|\iint \langle v, \phi_\epsilon \rangle  \; dx \, dt \bigg|
&\lesssim \|u \|_{L^\infty(I; L^2(Q))} \| \phi_{\epsilon }\|_{L^p(I; L^{2}(Q))}.
\end{align*}
The Caccioppoli inequality (Lemma~3.2 in \cite{KL} with $a=0$) yields
\begin{align*}
\|u \|_{L^\infty(I; L^2(Q))} \lesssim \|u \|_{L^2(I \times Q)} + \|u \|_{L^p(I; W^{1,p}(Q))}.
\end{align*}
Since $p \geq \frac{2d}{d + 2}$, we have the Sobolev embedding $W^{1,p}(Q) \subseteq L^{2}(Q)$. Thus,
\begin{align*}
\|\phi_{\epsilon }\|_{L^p(I; L^{2}(Q))} \lesssim \|\phi_{\epsilon }\|_{L^p(I; W^{1,p}(Q))} \leq \|\phi\|_{L^p(\re; W^{1,p}(\rd))} = 1.
\end{align*}
Altogether, we have found that
\begin{align}
\label{tdervbound.lem.eq1}
\bigg|\iint \langle v_\epsilon, \phi \rangle  \; dx \, dt \bigg| \lesssim \|u \|_{L^2(I \times Q)} + \|u \|_{L^p(I; W^{1,p}(Q))}.
\end{align}
Next, we get from the equation for $u$, using the summation convention for $i=1,\ldots, N$ and omitting the variable of integration $dx\, dt$ for the sake of readability, 
\begin{align*}
-\iint \langle \partial_t (v_\epsilon), \phi \rangle 
	&= \iint \bigg(\langle u, \partial_t (\chi   \phi_\epsilon)\rangle -  \langle u, (\partial_t\chi)   \phi_{\epsilon} \rangle \bigg)  \\
	&= \iint \bigg(\langle A_i(t,x,\nabla u), \nabla (\chi  \phi_{\epsilon})_i \rangle - B_i(t,x,\nabla u)(\chi   \phi_{\epsilon} )_i 
			- \langle u, (\partial_t\chi) \phi_{\epsilon}\rangle \bigg)  \\
    &= \iint \langle  (\chi A_i(t,x,\nabla u))_{\epsilon}, \nabla \phi_i \rangle
			+ \iint \langle A_i(t,x,\nabla u), (\nabla \chi)  (\phi_{\epsilon})_i \rangle  \\ 
    &\quad - \iint B_i(t,x,\nabla u)(\chi \phi_{\epsilon})_i 
 			- \iint \langle u, (\partial_t\chi)   \phi_{\epsilon} \rangle  \\ 
 	&=: \mathrm{I} + \mathrm{II} - \mathrm{III} - \mathrm{IV}.
\end{align*}
Using H\"older's inequality and the upper bound for $A$, we have
\[
|\mathrm{I}| 
	\lesssim (c_4 + c_1\|\nabla u\|_{L^p(I \times Q)}^{p-1}) \| \nabla \phi\|_{L^p(I \times Q)} \leq c_4 + c_1\|\nabla u\|_{L^p(I \times Q)}^{p-1}.
\]
Similarly, replacing $\nabla \phi$ by $\phi$, we get
\[
|\mathrm{II}| + |\mathrm{III}| \lesssim c_4 + (c_1+c_2)\|\nabla u\|_{L^p(I \times Q)}^{p-1}.
\]
For $\mathrm{IV}$, we can argue as in \eqref{tdervbound.lem.eq1} with $\phi_\epsilon$ replaced by $(\partial_t \chi)\phi_\epsilon$, in order to give 
\[
|\mathrm{IV}| \lesssim \|u \|_{L^2(I \times Q)} + \|u \|_{L^p(I; W^{1,p}(Q))}.
\]
Summarizing these estimates, we get
\[
\bigg|\iint \langle v_\epsilon, \phi \rangle  \; dx \, dt \bigg| + \bigg|\iint \langle \partial_t (v_\epsilon), \phi \rangle  \; dx \, dt \bigg| \leq C
%(\|u \|_{L^2(I \times Q)} + \|u \|_{L^p(I; W^{1,p}(Q))}).
\]
This is true for any $\phi \in \cS(\rdd)$ normalized in $L^p(\re; W^{1,p}(\rd))$. In view of the equivalence of Sobolev and potential spaces on $\rd$, this is the same as taking $\phi = J_x^1 \psi$, where $\psi \in \cS(\rdd)$ is normalized in $L^p(\rdd)$. Hence, we get
\[
\bigg|\iint \langle J_x^{1}v_\epsilon, \psi \rangle  \; dx \, dt \bigg| + \bigg|\iint \langle \partial_t (J_x^1 v_\epsilon), \psi \rangle  \; dx \, dt \bigg| \leq C
%(\|u \|_{L^2(I \times Q)} + \|u \|_{L^p(I; W^{1,p}(Q))}).
\]
Since  $\cS(\rdd)$ is dense in $L^{p'}(\rdd)$, we obtain
\[
\|J_x^1 v_\eps\|_{L^p(\rdd)} + \|\partial_t (J_x^1 v_\eps)\|_{L^p(\rdd)} \leq C
%(\|u \|_{L^2(I \times Q)} + \|u \|_{L^p(I; W^{1,p}(Q))}).
\]
Now, we invoke the equivalence of Sobolev and potential spaces in $t$ and apply Fubini's theorem to conclude the bound for $\|J_t^{-1} J_x^1 (v_\epsilon)\|_{L^p(\rdd)}$.
\end{proof}

%%%%%%
\section{Interpolation estimate of a mixed potentials}

\noindent We begin by recalling (a simple version of) the Mihlin multiplier theorem.

\begin{proposition}[Theorem 8.2 in {\cite{M-S}}]
\label{multthrm.thrm}
Let $n\geq 1$, let $m: \mathbb{R}^{n} \to \mathbb{C}$ satisfy, for all multi-index of order $|\alpha| \leq n+2$ and all $\xi \neq 0$,
\[
|\partial^\alpha_{\xi}m(\xi)| \le M  |\xi|^{-|\alpha|}.
\]
Then, for any $q \in (1,\infty)$ there is a constant $C=C(n,q)$, such that for all $\phi \in \cS(\re^n)$ and for $\cF$ the Fourier transform on $\cS(\re^n)$,
\[
\|\cF^{-1}(m \cF \phi)\|_{L^q(\re^n)} \leq CM \|\phi\|_{L^q(\re^n)}.
\]
\end{proposition}

The multiplier theorem entails quantitative bounds for Bessel potentials.

\begin{lemma}
\label{BesselBound.lem}
Let $a \in [0,\infty)$ and $b \in \re$. For all $q \in (1,\infty)$ there is a constant $C = C(d,q)$ such that for all $\phi \in \cS(\rd)$,
\[
\|J_x^{2a+2\i b} \phi \|_{L^q(\rd)} \leq C (1+a+|b|)^{d+2} \|\phi\|_{L^q(\rd)}.
\]
The same holds for $J_t^{2a + 2\i b}$ upon replacing $d$ by $1$.
%In dimension one, the same holds for potentials $J_t^{2a+2\i b}$.
\end{lemma}

\begin{proof}
We put $h(\sigma) = (1+\sigma)^{-a-\i b}$. According to the Mihlin multiplier theorem, we have
\[
\|J_x^{2a+2\i b} \phi \|_{L^q(\rd)} \leq C M \|\phi\|_{L^q(\rd)},
\]
where
\[
M = \max_{|\alpha| \leq d+2} \sup_{\xi \neq 0} |\xi|^{|\alpha|} |\partial_\xi^\alpha(h(|\xi|^2))|.
\]
Let $\alpha$ be any multi-index. By induction on the length of $\alpha$, we find numerical constants $c_\beta(\alpha)$, one for each multi-index $\beta$ with $2 \beta_i \leq \alpha_i$, $i=1,\ldots,d$, such that
\[
\partial_\xi^\alpha(h(|\xi|^2)) = \sum_\beta c_\beta(\alpha) h^{(|\alpha|-|\beta|)}(|\xi|^2) \xi^{\alpha-2\beta}.
\]
Since the higher order derivatives of $h$ satisfy 
\begin{align*}
|h^{(k)}(\sigma)|
\leq c(k)(1+a+|b|)^k(1+\sigma)^{-a-k},
\end{align*}
we can take $M= C(d)(1+a+|b|)^{d+2}$. The one dimensional case is clearly included in the computation.
\end{proof}

We deduce the following interpolation inequality. In the proof we shall use the notion of holomorphic functions $f: \Omega \subset \mathbb{C} \to L^2(\rdd; \mathbb{C}^N)$. Holomorphy is defined via convergence of difference quotients. If $f$ is locally bounded, then it is equivalent to holomorphy of $z \mapsto \iint f(z) \overline{\phi} \; dx \, dt$ for all $\phi \in \cS(\rdd; \mathbb{C}^N)$. The reader can refer to Appendix~A of \cite{ABHN} for further background.

\begin{proposition}
\label{besselinterp.prop}
Let $f \in \cS(\rdd; \mathbb{C}^N)$, let $\theta \in (0,1)$ and let $q_0,q_{\theta}, q_1 \in (1,\infty)$ be such that $1/{q_{\theta}} = (1-\theta)/{q_0} + \theta/{q_1}$. Then there is a constant $C = C(d,N,q_0,q_1)$ such that
\[ 
\| J^{2\theta - 1}_xJ_t^{-\theta} f \|_{L^{q_{\theta}}(\mathbb{R}^{d+1})} \leq  C\| J_x^{-1}f \|_{L^{q_0}(\mathbb{R}^{d+1})}^{1-\theta} \|J_x^{1} J_t^{-1} f  \|_{L^{q_1}(\mathbb{R}^{d+1})}^{\theta}.
\]
\end{proposition}
\begin{proof}
By duality, we have
\begin{align}
\| J^{2\theta - 1}_xJ_t^{-\theta} f \|_{L^{q_{\theta}}(\mathbb{R}^{d+1})} = \sup_\phi \bigg| \iint \langle ( J^{2\theta - 1}_xJ_t^{-\theta}f)(x,t), \overline{\phi(x,t)} \rangle \; dx \, dt \bigg|,
\end{align}
where the supremum is taken over all $\phi \in \cS(\rdd; \mathbb{C}^N)$ normalized in $L^{q_{\theta}'}(\mathbb{R}^{d+1})$.
The idea of proof, coming from the Riesz--Thorin theorem, is to use a holomorphic parametrization of the duality pairing for fixed $\phi$ via functions defined on the strip $S:= \{a+\i b: a \in (0,1),\, b \in \re\}$. More precisely, we define whenever $z \in \overline{S}$,
\begin{align*}
F(z) := \e^{(z-\theta)^2} J_x^{2z-1} J_t^{-z} f, \qquad
G(z) := |\phi|^{\frac{(1-z)q_{\theta}'}{q_0'}+ \frac{zq_{\theta}'}{q_1'}} \frac{\overline{\phi}}{|\phi|},
\end{align*}
where the expression for $G(z)$ is interpreted as $0$ on the set where $\phi$ vanishes.

We derive properties of $F$. For $z = a + \i b$ we have
\begin{align}
\label{besselinterp.prop.eq1}
F(z) = \Big(\e^{(a-\theta)^2-b^2} \e^{\i 2b(a-\theta)}\Big) J_x^{2a + 2 \i b} J_t^{(1-a)-\i b}(J_x^{-1} J_t^{-1} f).
\end{align}
Since $J_x^{-1} J_t^{-1} f$ is a Schwartz function, Lemma~\ref{BesselBound.lem} applied componentwise in combination with Fubini's theorem yields that $F$ is {\it qualitatively} bounded on $S$ with values in $L^2(\rdd; \mathbb{C}^N)$. (The polynomial growth in $b$ is compensated by the exponential function.) Again using Lemma~\ref{BesselBound.lem}, we have the following quantitative bounds on $\partial S$:
\begin{align}
\label{besselinterp.prop.eq2}
\begin{split}
\|F(\i b)\|_{L^{q_0}(\rdd)} 
&\leq C(d,N,q_0) \|J_x^{-1} f\|_{L^{q_0}(\rdd)},\\
\|F(1+ \i b)\|_{L^{q_1}(\rdd)} 
&\leq C(d,N,q_1) \|J_t^{-1}J_x^{1} f\|_{L^{q_1}(\rdd)}.
\end{split}
\end{align}
Next, it follows from $\mathcal{F}f \in \cS(\rdd; \mathbb{C}^N)$ and dominated convergence, that we have a continuous function
\begin{align*}
\mathcal{F}F: \overline{S} \to L^2(\rdd; \mathbb{C}^N), \qquad z \mapsto \e^{(z-\theta)^2} (1+|\xi|^2)^{\frac{1-2z}{2}} (1+|\tau|^2)^{\frac{z}{2}} \mathcal{F}f(\tau,\xi).
\end{align*}
Since the Fourier transform is isometric on $L^2(\rdd; \mathbb{C}^N)$, the same follows for $F$.
Finally, $F$ is holomorphic in $S$. Indeed, for any $\psi \in \cS(\rdd; \mathbb{C}^N)$ we can use Parseval's formula to give
\begin{align*}
\iint \langle F(z), \overline{\psi} \rangle \; dx \, dt = \iint \big \langle e^{(z - \theta)^2} (1 +|\xi|^2)^{\frac{1-2z}{2}}(1 +|\tau|^2)^{\frac{z}{2}} \mathcal{F}f, \overline{\mathcal{F}\psi} \big \rangle \; d\xi \, d\tau
\end{align*}
and the integral in $z$ along any triangle $\triangle \Subset S$ vanishes by Fubini's theorem and holomorphy of the integrand for fixed $(\tau,\xi)$.

The function $G: \overline{S} \to L^2(\rdd; \mathbb{C}^N)$ enjoys the same kind of properties. Here, boundedness follows directly from $\phi \in \cS(\rdd; \mathbb{C}^N)$, continuity and holomorphy are obtained as before, and on $\partial S$ we get from H\"older's inequality and the normalization of $\phi$ that
\begin{align}
\label{besselinterp.prop.eq3}
\begin{split}
\|G(\i b)\|_{L^{q'_0}(\rdd)} 
\leq 1, \qquad
\|G(1+ \i b)\|_{L^{q'_1}(\rdd)} 
\leq 1.
\end{split}
\end{align}

Now, define a scalar-valued function on $\overline{S}$ by
\begin{align*}
H(z) := \iint \langle F(z), G(z) \rangle \; d x \, dt.
\end{align*}
The $L^2(\rdd; \mathbb{C}^N)$-valued properties for $F$ and $G$ above imply that $H$ is bounded and continuous on $\overline{S}$ and holomorphic in $S$. (The inner product preserves continuity and holomorphy by nearly the same proof as for products of scalar functions.)
On the boundary, we conclude from \eqref{besselinterp.prop.eq2}, \eqref{besselinterp.prop.eq3} and H\"older's inequality that
\begin{align*}
 |H(\i b)| &\leq C(d,N,q_0) \|J_x^{-1} f\|_{L^{q_0}(\rdd)} =: M_0,\\
 |H(1+\i b)| &\leq C(d,N,q_1) \|J_t^{-1}J_x^{1} f\|_{L^{q_1}(\rdd)}=:M_1.
\end{align*}
Hadamard's three lines theorem yields $|H(\theta)| \leq M_0^{1-\theta} M_1^{\theta}$. This means that
\begin{align*}
\bigg|\iint \langle J^{2\theta - 1}_xJ_t^{-\theta}f, \overline{\phi} \, \rangle \; dx \, dt \bigg|  \leq C \|J_x^{-1} f\|_{L^{q_0}(\rdd)}^{1-\theta} \|J_t^{-1}J_x^{1} f\|_{L^{q_1}(\rdd)}^{\theta}
\end{align*}
and the  claim follows since $\phi$ was arbitrary and normalized in $L^{q_{\theta}'}(\mathbb{R}^{d+1})$.
\end{proof}

\section{Proof of Theorem~\ref{mainthrm.thrm}}
\label{proof.sec}

\noindent Let $u$ be a weak solution to \eqref{pparatype.eq} in $I \times Q$. We define nested intervals and cubes $I'' \Subset I' \Subset I$ and $Q'' \Subset Q' \Subset Q$ and pick a smooth function $1_{I'' \times Q''} \le \chi \le 1_{I'\times Q'}$. Define a localized version $v = \chi u$ as in Lemma~\ref{tdervbound.lem}. Since the nested sets are arbitrary, it suffices to obtain the continuity statements on $I'' \times Q''$.

\noindent \emph{Step 1: H\"older continuity with values in spatial $L^q$}. Lemma~\ref{tdervbound.lem} justifies applying Proposition~\ref{besselinterp.prop} to $f=v_\epsilon$ with $q_0 = p+\delta$ and $q_1 = p'$. Thus, setting $\theta =\frac{1}{2}$ in Proposition~\ref{besselinterp.prop}, we have
\begin{align}
\label{half-order-exponent.eq1}
\| J_t^{-1/2} v_\epsilon \|_{L^{q}(\mathbb{R}^{d+1})} \leq C,
\end{align}
for some admissible $C$. The exponent $q>2$ is given by
\begin{align}
\label{half-order-exponent.eq2}
 \frac{1}{q} = \frac{1}{2} + \frac{1}{2} \left(\frac{1}{p+\delta} - \frac{1}{p} \right).
\end{align}
We have $(J_t^{-1/2} v_\epsilon)(\cdot \,,x) = J_t^{-1/2} (v_\epsilon(\cdot \,,x)) \in L^q(\re)$ for every $x \in \rd$ since $v_\epsilon$ is a Schwartz function, but \eqref{half-order-exponent.eq1} gives a quantitative bound. 

Fix $x \in \rd$. As $q' < 2$, we obtain from Minkowski's inequality that $G_t^{1/2} \in L^{q'}(\re)$. This Bessel kernel was defined in \eqref{potential-formula.eq}. Hence, we have by Young's convolution inequality that
\begin{align}
\label{half-order-exponent.eq3}
\|v_\epsilon(\cdot \,,x)\|_{L^\infty(\re)}
\leq \|G_t^{1/2}\|_{L^{q'}(\re)} \|J_t^{-1/2} v_\epsilon(\cdot \,,x)\|_{L^q(\re)}.
\end{align}
Now, we appeal to a Fourier analytic characterization of H\"older continuity. This uses a smooth function $\psi$ with support in the set $\{\frac{1}{2} \leq t \leq 4\}$ with the property that $\sum_{j \in \mathbb{Z}} \psi_j(t) = 1$ for all $t \neq 0$, where $\psi_j(t) = \psi(2^{-j}t)$. For one construction see Lemma 8.1 in \cite{M-S}.

\begin{lemma}[Lemma~8.6 in {\cite{M-S}}]
\label{holder.lem}
Let $f \in \cS(\re)$ and let $0< \alpha <1$. There is a constant $C = C(\alpha)$ such that for all $t \neq s$,
\begin{align*}
\frac{|f(t) - f(s)|}{|t-s|^\alpha} \leq C \sup_{j \in \mathbb{Z}} 2^{j\alpha} \|\cF_t^{-1}(\psi_j \cF_t f)\|_{L^\infty(\re)}.
\end{align*}
\end{lemma}

We put $f := v_\epsilon(\cdot \,,x)$. Since $\cF_t^{-1} (\psi_j \cF_t f)$ has a Fourier transform with support in $(-2^{j+2}, 2^{j+2})$, Bernstein's inequality (Lemma 4.13 in \cite{M-S}) yields
\begin{align*}
\|\cF_t^{-1}(\psi_j \cF_t f)\|_{L^\infty(\re)} 
&\leq c 2^{j/q} \|\cF_t^{-1}(\psi_j \cF_t f)\|_{L^q(\re)} \\
&\leq c 2^{j/q} C(p,\psi) 2^{\mathrm{min}\{-j/2,1 \}} \|J_t^{-1/2} v_\epsilon(\cdot\, ,x)\|_{L^q(\re)},
\end{align*}
where $c$ is a numerical constant and the second step is due to the Mihlin multiplier theorem applied to $m(\tau) = \psi_j(\tau) (1+|\tau|^2)^{-1/4}$. The computation of the Mihlin norm is done \emph{verbatim} as in the proof of Lemma~\ref{BesselBound.lem}, taking into account that on the support of $\psi_j$ we have $2^{j-1} \leq |\tau| \leq 2^{j+2}$ in order to obtain the decay in $j$.

%Indeed,
%\[ |\tau|^{k} |\partial_{\tau}^{k} ( \psi_j(\tau) (1+|\tau|^2)^{-1/4} )| \le c  2^{\mathrm{min}\{-j/2,1 \}}  \]
%for all integrers $k \ge 0$ follows from taking into account the support of $\psi_j$. 

The assumptions of Theorem \ref{multthrm.thrm} are hence satisfied. Consequently, we can take $\alpha = \frac{1}{2} - \frac{1}{q} = \frac{1}{2} (\tfrac{1}{p} - \tfrac{1}{p+\delta})$ in Lemma~\ref{holder.lem}. In view of \eqref{half-order-exponent.eq3} we find for all $t \neq s$ and all $x \in \rd$ that
\begin{align}
\label{half-order-exponent.eq4}
|v_\epsilon(t,x)| + \frac{|v_\epsilon(t,x) - v_\epsilon(s,x)|}{|t-s|^\alpha} \leq C \|J_t^{-1/2} v_\epsilon(\cdot \,,x)\|_{L^q(\re)}.
\end{align}
We fix a representative for $v$, a subsequence of $\epsilon$ and a set $E$ of $(d+1)$-measure zero such that $v_\epsilon(t,x) \to v(t,x)$
as $\epsilon \to 0$, whenever $(t,x) \in \rdd \setminus E$. Then we integrate the $q$-th power in $x$, use \eqref{half-order-exponent.eq1} and pass to the limit via Fatou's lemma, to get
\begin{align*}
\sup_{t \in \re} \left(\int_{\rd} |v(t,x)|^{q}\right)^{1/q} + \sup_{\substack{t,s \in \re \\ t \neq s}}\left(\int_{\rd}\frac{|v(t,x) -v(s,x)|^{q}}{|t - s|^{\alpha q}} \right)^{1/q} \leq C.
\end{align*}
Since $v=u$ on $I'' \times Q''$, we obtain $u \in C^\alpha(I''; L^q(Q''))$ as required.

\noindent \emph{Step 2: H\"older continuity on almost all segments in time}. Since the Fourier transform turns convolutions into products, we obtain for all $\varphi \in \cS(\rdd)$ that
\begin{align}
\label{half-order-exponent.eq5}
\langle J_t^{-1/2} v_\epsilon, \varphi \rangle
= \langle J_t^{-1/2} v, \varphi_\epsilon \rangle,
\end{align}
where we use the duality pairing on $\cS'(\rdd)$. In the limit as $\epsilon \to 0$, we have $\varphi_\epsilon \to \varphi$ in $\cS(\rdd)$ and therefore $J_t^{-1/2} v_\epsilon \to J_t^{-1/2} v$ in $\cS'(\rdd)$. On the other hand, this sequence is bounded in $L^q(\rdd)$ by \eqref{half-order-exponent.eq1} and hence admits a weakly convergent subsequence. Identifying the limits, we get $J_t^{-1/2} v \in L^q(\rdd)$. Now, we can use \eqref{half-order-exponent.eq5} to write $J_t^{-1/2} v_\epsilon = (J_t^{-1/2} v)_\epsilon$ and obtain strong convergence in $L^q(\rdd)$. This implies that $\|J_t^{-1/2} v_\epsilon\|_{L^q(\re)} \to \|J_t^{-1/2} v\|_{L^q(\re)}$ in $L^q(\rd)$. Hence, we can pass to a subsequence such that for almost every $x \in \rd$,
\begin{align*}
\|J_t^{-1/2} v_\epsilon(\cdot \,,x)\|_{L^q(\re)} \to \|J_t^{-1/2} v(\cdot \,,x)\|_{L^q(\re)}.
\end{align*}
Fix $x$ with this property and such that $E_x := \{t : (t,x) \in E\}$ has $1$-measure zero, where $E$ is as in Step~1. Passing to the limit in \eqref{half-order-exponent.eq4}, we obtain that $v(\cdot\,,x)$ satisfies the $\alpha$-H\"older condition on $E_x$. Hence, we can re-define $v(\cdot\,,x)$ so that it is $\alpha$-H\"older continuous. Since all modifications take place in $E$, we obtain a representative for $v$ that is $\alpha$-H\"older continuous on $\re \times \{x\}$ for a.e.\ $x \in \rd$. We conclude again since $v=u$ on $I'' \times Q''$. \hfill {\qedsymbol}
%%%%%%%%%%%%%%%%%%%%
\def\cprime{$'$} \def\cprime{$'$} \def\cprime{$'$}


\begin{thebibliography}{100}
\providecommand{\url}[1]{{\tt #1}}
\providecommand{\urlprefix}{URL}
\providecommand{\eprint}[2][]{\url{#2}}

\bibitem{ABHN}
\textsc{W. Arendt}, \textsc{C. Batty}, \textsc{M. Hieber} and \textsc{F. Neubrander}. 
\newblock Vector-valued Laplace transforms and Cauchy problems. Second edition. 
Monographs in Mathematics, vol.~96. Birkh\"{a}user/Springer, Basel, 2011.

\bibitem{ABES1}
\textsc{P. Auscher, S. Bortz, M. Egert and O. Saari}.
On regularity of weak solutions to linear parabolic systems with measurable coefficients. 
{\it J. Math. Pures Appl.} (9) 121 (2019), 216-243.

\bibitem{Denk-Kaip}
\textsc{R. Denk} and \textsc{M. Kaip}.
General parabolic mixed order systems in {${L_p}$} and applications.
Operator Theory: Advances and Applications, vol.~239, Birkh\"{a}user/Springer, Cham, 2013.

\bibitem{GS1982}
\textsc{M.~Giaquinta} and \textsc{M.~Struwe}.
On the partial regularity of weak solutions of nonlinear parabolic systems.
\textit{Math. Z.} 179 (1982), no. 4, 437--451. 

\bibitem{KL}
\textsc{J. Kinnunen} and \textsc{J. Lewis.}
\newblock Higher integrability for parabolic systems of p-Laplacian type. 
{\it Duke Math. J.} 102 (2000), no. 2, 253-271. 

\bibitem{Li57}
\textsc{J.-L.~Lions}.
\newblock{ Sur les probl\`emes mixtes pour certains syst\`emes paraboliques dans des ouverts non cylindriques}. 
\textit{Ann. Inst. Fourier, Grenoble} 7 (1957), 143--182.

\bibitem{M-S}
\textsc{C. Muscalu} and \textsc{W. Schlag}.
\newblock Classical and multilinear harmonic analysis. {V}ol. {I}, Cambridge Studies in Advanced Mathematics, vol.~137, Cambridge University Press, Cambridge, 2013.

\bibitem{Stein}
\textsc{E.M.~Stein}.
\newblock Singular {I}ntegrals and {D}ifferentiability {P}roperties of
{F}unctions. Princeton Mathematical Series, no.~30,
\newblock Princeton University Press, Princeton NJ, 1970.

\bibitem{zaton}
\textsc{W.~Zato\'n}.
\newblock {Tent space well-posedness for parabolic Cauchy problems with rough coefficents\/}.
Preprint (2019), \url{https://arxiv.org/abs/1909.12197}, to appear in \textit{J. Differential Equations}.
\end{thebibliography}
\end{document}